\documentclass[11pt]{amsart}
\usepackage{amssymb}
\usepackage{amsthm}
\usepackage{mathrsfs}
\usepackage{aliascnt}
\usepackage{latexsym,amsfonts,amssymb,amsmath,amscd}
\usepackage{graphicx,color,ifpdf}
\usepackage{fancyhdr}
\allowdisplaybreaks
\usepackage[linktocpage, colorlinks, backref = page, citecolor = black, linkcolor = black]{hyperref}

\newcommand {\diam }{\mathrm{diam}}
\newcommand {\Isom}{\mathrm{Isom}}

\newcommand {\vol }{\mathrm{vol}}

\newcommand {\R}{\mathbb{R}}

\newcommand {\grad }{\triangledown}
\newcommand {\Ric}{\mathrm{Ric}}
\newcommand {\spa}[1]{\langle{#1}\rangle}
\newcommand {\myarrow}[1]{\mathop{\longrightarrow}\limits^{#1}}
\newcommand{\GH}{\myarrow{GH}}

\newcommand{\XXint}[3]{{
		\setbox0=\hbox{$#1{#2#3}{\int}$}
		\vcenter{\hbox{$#2#3$}}\kern-.5\wd0}}

\input xy
\xyoption{all}

\numberwithin{equation}{section}
\newtheorem{proposition}{Proposition}
\newtheorem{lemma}[proposition]{Lemma}
\newtheorem{sublemma}[proposition]{Sublemma}
\newtheorem{theorem}[proposition]{Theorem}
\newtheorem{corollary}[proposition]{Corollary}

\numberwithin{proposition}{section}
\newtheorem{problem}[proposition]{Problem}

\theoremstyle{definition}
\newtheorem{definition}[proposition]{Definition}

\newtheorem{remark}[proposition]{Remark}

\title[Finite generation of fundamental groups]{Finite generation of fundamental groups for manifolds with nonnegative Ricci curvature whose universal cover is almost $k$-polar at infinity}
\author{HongZhi Huang}
\address{ Department of Mathematics \\  Jinan University\\ Guangzhou 510632\\ PR China}
\email{\href{mailto:huanghz@jnu.edu.cn
	}{huanghz@jnu.edu.cn,
}{\href{mailto:hyyqsaax@163.com
}{hyyqsaax@163.com
}}}

\begin{document}

\maketitle
\begin{abstract}

In this article, we prove that the fundamental group $\pi_1(M)$ of a complete open manifold $M$ with nonnegative Ricci curvature is finitely generated, under the condition that the Riemannian universal cover $\tilde M$ satisfies an "almost $k$-polar at infinity" condition. Additionally, such $\pi_1(M)$ is virtually abelian. Furthermore, we demonstrate that the base point of any tangent cone at infinity of such a manifold is nearly a pole. In the case where $\tilde M$ exhibits almost maximal Euclidean volume growth, we prove that $M$ deformation retracts to a closed submanifold $F$ which is diffeomorphic to a flat manifold, provided $M$ is not simply connected.

\vspace*{5pt}

\vspace*{5pt}

\end{abstract}

\section{introduction}

The Milnor conjecture proposes that the fundamental group of a complete open manifold $M$ with nonnegative Ricci curvature is finitely generated \cite{Mi68}. This conjecture remained unresolved for decades until recently Brue-Naber-Semola provided counterexamples in dimensions $n\ge 6$ \cite{BAD23a,BAD23b}, thereby disproving the conjecture. Nevertheless, the search for sufficient conditions ensuring the finite generation of the fundamental group in such manifolds remains of interest. Efforts in this direction have led to several important discoveries. For instance, manifolds $M$ exhibiting Euclidean volume growth have a finite fundamental group \cite{Li86}; Other conditions supporting finite generation include when the volume of the $n$-dimensional manifold $M$ exhibits polynomial growth of order $\ge(n-k)$ and the first Betti number $\ge k$ \cite{An90}, or when the base point of every tangent cone at infinity of $M$ is nearly a pole \cite{Sor99}. 

Another strategy to establish the finitely generated nature of the fundamental group is to impose conditions on the Riemannian universal cover $\tilde{M}$ rather than $M$, which brings new challenges. For example, it is still unclear whether the Euclidean volume growth of $\tilde{M}$ leads to the finite generation of the fundamental group, in contrast to the aforementioned result in \cite{Li86}. In this pursuit, Pan has confirmed the Milnor conjecture under specific geometric stability conditions applied to the universal cover at infinity, establishing the following results:

\begin{theorem}\cite{Pan19a,Pan19b}\label{PansResult}
	Let $M$ be a complete open manifold  with nonnegative Ricci curvature. If the Riemannian universal cover $\tilde M$ satisfies one of the following conditions, then $\pi_1(M)$ is finitely generated,
	\begin{enumerate}
		\item (\cite[Theorem A]{Pan19a}) Any tangent cone at infinity of $\tilde M$ is a metric cone, whose maximal Euclidean factor has dimension $k\ge 0$.
		
		\item (\cite[Theorem B (1)]{Pan19b}) There exists a compact metric space $K$, such that any tangent cone at infinity of $\tilde M$ is a metric cone, whose cross-section is $\epsilon_K$-close to $K$ with respect to the Gromov-Hausdorff distance, where $\epsilon_K>0$ is a universal constant only depending on $K$.
	\end{enumerate}   
\end{theorem}

The above theorem covers many interesting cases, including when $M$ satisfies nonnegative sectional curvature, or when $\tilde M$ exhibits almost maximal volume growth, or when $\tilde M$ has both Euclidean volume growth and a unique tangent cone at infinity. 

In this article, our main theorem affirms the Milnor conjecture under a geometric stability assumption on $\tilde M$, called almost $k$-polar at infinity, which unifies and relaxes the aforementioned two conditions proposed by Pan. Additionally, our approach differs from previous ones and appears to be more concise, largely due to the transformation theorem for Cheeger-Colding's almost splitting maps (\cite{CJN21,HH22}, as seen in Theorem \ref{TransThm} below).

Before presenting our result, we elucidate some concepts. For a manifold $M$ and $p\in M$, we say that an $r$-geodesic ball $B_r(p)$ is $(\epsilon,k)$-Euclidean for some small $\epsilon>0$ and integer $k\ge0$, if there exists a pointed metric product $(\R^k\times X,(0^k,x^*))$, where $X$ is a metric space, such that the Gromov-Hausdorff distance $d_{GH}(B_r(p),B_r((0^k,x^*)))\le\epsilon r$. Given small constants $\delta,\eta>0$ and integer $k\ge 0$, we define the meaning of $(\delta,\eta,k)$-polar at infinity as follows.

\begin{definition}\label{DefAlmostk-Polar} 
	If there exist $p\in M$ and $R>0$, such that for any $r\ge R$, the following conditions hold,
	\begin{enumerate}
		\item 	$B_r( p)$ is $(\delta,k)$-Euclidean and not $(\eta,k+1)$-Euclidean,
		
		\item  for any $x\in\partial B_r(p)$, $d(x,\partial B_{2r}(p))\le(1+\delta )r$,
	\end{enumerate}
	then we say that $M$ satisfies $(\delta,\eta,k)$-polar at infinity (with respect to $p$).
\end{definition}

Note that, in either case (1) or (2) of Theorem \ref{PansResult}, $\tilde M$ satisfies the $(\delta, \eta, k)$-polar condition at infinity for suitable values of $\delta$, $\eta$, and $k$. Specifically, in case (1), $\eta$ depends solely on $n$, and $\delta$ can be arbitrarily small, whereas in case (2), both $k$ and $\eta$ depend on $K$, with $\delta$ set at $2\epsilon_K$.

Our main result is stated below.

\begin{theorem}\label{MainThm}
	For any integer $n\ge 2$ and small $\eta>0$, there exists a constant $\delta=\delta(n,\eta)$ only depending on $n,\eta$, to the following effect. Let $M$ be a complete open $n$-manifold with nonnegative Ricci curvature. If the Riemannian universal cover $\tilde M$ satisfies $(\delta,\eta,k)$-polar at infinity for some $k\in[0,n]$, then $\pi_1(M)$ is finitely generated.

\end{theorem}

As mentioned before, Theorem \ref{MainThm} implies Theorem \ref{PansResult}. Moreover, as is well-known, the tangent cones at infinity of an open manifold with nonnegative Ricci curvature are far from unique. As demonstrated in \cite{CN13}, there is a great deal of freedom in selecting our tangent cones at infinity. Thus Theorem \ref{MainThm} covers further examples, such as the case that $\tilde M$ possesses two tangent cones at infinity with different numbers of maximal Euclidean factor, which are not close to each other in the Gromov-Hausdorff sense.

For $k$ in Theorem \ref{MainThm} large, we have the following corollary.

\begin{corollary}\label{Largek}
	For any $n\ge 2$ and $v>0$, there exist constants $\delta_1(n,v)$, $\delta_2(n)$ to the following effect. Let $M$ be a complete open $n$-manifold with nonnegative Ricci curvature. If the Riemannian universal cover $\tilde M$ and $\tilde p\in\tilde M$ satisfies either of the following conditions, 
	\begin{enumerate}
		\item the volume $\vol(B_r(\tilde p))\ge vr^n$, and $B_r(\tilde p)$ is $(\delta_1(n,v),n-2)$-Euclidean.

		\item $B_r(\tilde p)$ is $(\delta_2(n),n-1)$-Euclidean.

	\end{enumerate}
	for any large $r>0$, then $\pi_1(M)$ is finitely generated.
\end{corollary}

\begin{remark}
	By employing the approach of Theorem \ref{MainThm}, according to \cite[Proposition 4.2]{Pan22}, it can be shown that for any $\epsilon>0$, there exists $\delta=\delta(n,\eta,\epsilon)$ such that a space $(M,p)$ satisfying the conditions of Theorem \ref{MainThm} or Corollary \ref{Largek} exhibits an escape rate less than $\epsilon$, as defined in \cite{Pan22}. The detailed proof for this has been included in the appendix. Therefore, according to the main theorem in \cite{Pan22}, $\pi_1(M)$ in Theorem \ref{MainThm} or Corollary \ref{Largek} is virtually abelian.
	
\end{remark}

	Recall that in a metric space $Z$, a point $z\in Z$ is called a pole of $Z$, if for any $z'\in Z\backslash\{z\}$, there exists a ray starting from $z$ and passing through $z'$. Our next result states that the base point of any tangent cone at infinity of $M$ in Theorem \ref{MainThm} is nearly a pole. 
	
	\begin{theorem}\label{MisalsmotkPolar}
		For any integer $n\ge 2$ and small $\eta,\epsilon>0$, there exists $\delta=\delta(n,\eta,\epsilon)>0$ to the following effect. Let $M$ be a manifold as described in Theorem \ref{MainThm}. Then for any tangent cone at infinity, $(Y,y^*)$, of $M$, there exists a length space $(Z,z)$ with a pole at $z$ such that $d_{GH}((Y,y^*),(Z,z))\le\epsilon$.

	\end{theorem}

	The above theorem clarifies that the class of manifolds in Theorem \ref{MainThm} is a subsets of the manifolds in Sormani's Pole Group theorem \cite[Theorem 11]{Sor99}. Hence, Theorem \ref{MainThm} can be viewed as a corollary of Sormani's result combining Theorem \ref{MisalsmotkPolar}. However, our proof of Theorem \ref{MisalsmotkPolar} logically relies on the conclusion of Theorem \ref{MainThm}.

	A natural sufficient condition for $\tilde M$ to satisfy almost $k$-polar at infinity is that $\tilde M$ exhibits almost maximal volume growth. In such an extreme case, \cite[Theorem A]{Pan19b} claims that the fundamental group of $M$ is finitely generated and virtually abelian. Building upon Theorem \ref{MisalsmotkPolar}, we can derive an improved version as follows.
	
	\begin{theorem}\label{MaximalVolGrowth}
		For any integer $n\ge 2$, there exists a constant $\delta=\delta(n)$, to the following effect. Let $M$ be a complete open $n$-manifold with nonnegative Ricci curvature. If the Riemannian universal cover $\tilde M$ satisfies, for any $r>0$, $\vol(B_r(\tilde p))\ge (1-\delta)\vol(B_r(0^n))$, then $M$ deformation retracts to a closed submanifold $F$ which is diffeomorphic to a flat manifold, provided $M$ is not simply connected.
	\end{theorem}
	\begin{problem}
		The author is uncertain whether the conclusion of the above theorem can be strengthened to that $M$ is diffeomorphic to a flat open manifold.
	\end{problem}

	Before proceeding to the next section, we give a brief explanation of why Theorem \ref{MainThm} holds. If the conclusion of Theorem \ref{MainThm} is false, then there exists an infinite Gromov's short generator system $\{\gamma_1,...,\gamma_s,...\}\subset\pi_1(M,p)$. Let $\{r_1,...,r_s,...\}$ be the length spectrum of this basis. We observe that, if one additionally assumes $\pi_1(M)$ is abelian or nilpotent, then there exists a small $\delta>0$ not depending on $i$, satisfying that, $(\delta^{-1}r_i^{-1}M,p)$ almost splits one more $\R$-factor than $(r_i^{-1}M,p)$ does in the Gromov-Hausdorff sense. Another observation is that, under the almost stability assumption in Definition \ref{DefAlmostk-Polar} (1) on $\tilde M$, by a transformation theorem for almost splitting functions, this additional almost $\R$-factor will be preserved when scaling up the metric of $(\delta^{-1}r_i^{-1}M,p)$. Now if there exists $i_1<i$ such that $r_{i_1}<\delta r_i$, then we conclude that $(\delta^{-1}r_{i_1}^{-1}M,p)$ almost splits two additional Euclidean factors compared to $(r_i^{-1}M,p)$. By contradiction assumption, we may start the above argument from a sufficiently large $i$, such that we could repeat the above process $n+1$ times which will yield a contradiction. 
	
	We have organized the proofs of Theorem \ref{MainThm} and Corollary \ref{Largek} in Section \ref{ProfMain}, and the proofs of Theorem \ref{MisalsmotkPolar} and \ref{MaximalVolGrowth} in Section \ref{ProfotherThms}. The discussion of virtual abelianness and other consequences is included in the appendix.

\section{Preliminary}

In this paper, we employ $\Psi(x_1, \ldots, x_k|c_1, \ldots, c_s)$ to represent certain universal positive functions satisfying the condition $\Psi\to0$ as $x_1, \ldots, x_k \to 0$ while $c_1, \ldots, c_s$ remain fixed. Note that these functions may vary from line to line without explicit specification if there's no ambiguity.

In this section, we provide references to some key lemmas essential to our proof for the convenience of our readers.

Firstly, we start with the following reduction by Wilking.
\begin{theorem}\cite{Wil00}\label{WilkingsReduction}
	Let $M$ be an open manifold with nonnegative Ricci curvature. If $\pi_1(M)$ is not finitely generated, then it contains a non-finitely generated abelian subgroup.
\end{theorem}
Note that our Theorem \ref{MainThm} imposes restriction only on the universal cover $\tilde M$, hence by the above reduction, it suffices to prove Theorem \ref{MainThm} under an extra assumption that $\pi_1(M)$ is abelian.

The second key ingredient is an observation by Pan (\cite{Pan20}), which we restate in a slightly generalized form.

\begin{lemma}\label{Nonconnectness}
	Let $(M_i,p_i)$ be a sequence of open $n$-manifolds with nonnegative Ricci curvature, and $(\tilde M_i,\tilde p_i)\to(M_i,p_i)$ be the Riemannian universal cover. For each $i$, fix a set of Gromov's short generators $\{\gamma_{i,1},\gamma_{i,2},...\}\subset\pi_1(M_i)$ and let $r_{i,j}:=d_i(\tilde p_i,\gamma_{i,j}\tilde p_i)$, where $d_i$ is the distance on $\tilde M_i$. If $(Y,p,G)$ is an equivariant Gromov-Hausdorff limit of a subsequence of $(r_{i,j_i}^{-1}\tilde M_i,\tilde p_i,\Gamma_i)$, where $\Gamma_i=\pi_1(M_i,p_i)$ is the deck transformation, then the orbit $G(p)$ is not connected.
\end{lemma}
\begin{proof}
	The case for $j_i=1$ is trivial. Hence we may assume all $j_i>1$. Let $\Gamma_i^0$ be the subgroup of $\Gamma_i$ generated by $\{\gamma_{i,1},\gamma_{i,2},...,\gamma_{i,j_i-1}\}$. Up to a subsequence, we may assume the equivariant Gromov-Hausdorff convergence $(r_{i,j_i}^{-1}\tilde M_i,\tilde p_i,\Gamma_i^0)\GH(Y,p,G^0)$. For each fixed $i$ and any $\gamma_i\in\Gamma_i-\Gamma_i^0$, $d_i(\gamma_i\tilde p_i,\Gamma_i^0(\tilde p_i))\ge d_i(\gamma_{i,j_i}\tilde p_i,\tilde p_i)=r_{i,j_i}$. Otherwise, there would exist $\gamma_i'\in\Gamma_i^0$, such that $d_i(\gamma_i'\gamma_i\tilde p_i,\tilde p_i)< d_i(\gamma_{i,j_i}\tilde p_i,\tilde p_i)$, then according to the selection step of Gromov's short generators, the next item of $\gamma_{i,j_i-1}$ is not $\gamma_{i,j_i}$, which contradicts to our assumption. So we have $$r_{i,j_i}^{-1}d_i(\gamma_i\tilde p_i,\Gamma_i^0(\tilde p_i))\ge 1$$ for any $\gamma_i\in\Gamma_i-\Gamma_i^0$. This concludes that for any $\gamma\in G-G^0$, $d_\infty(\gamma p,G^0(p))\ge1$, where $d_\infty$ is the distance on $Y$. So $G^0(p)$ is an open subset of $G(p)$. As an equivariant Gromov-Hausdorff limit group, $G^0$ is closed. Combining $G(p)-G^0(p)\neq\emptyset$, we conclude $G(p)$ is not connected.

\end{proof}

At the end of this section, we present our crucial lemma, a transformation theorem for the splitting map. Recall that a Cheeger-Colding's $(\delta,k)$-splitting map is a harmonic map, $u=(u^1,...,u^k):B_r(p)\to\R^k$ satisfying that, for $\alpha,\beta=1,...,k$,
\begin{enumerate}
	\item  $\sup_{x\in B_r(p)}\|\grad u^\alpha(x)\|\le 1+\delta$,
	
	\item  $\frac{1}{\vol(B_r(p))}\int_{B_r(p)}|\spa{\grad u^{\alpha},\grad u^{\beta}}-\delta^{\alpha\beta}|\le \delta$,
	
\end{enumerate}	

Almost splitting maps serve as a tool to characterize the almost splitting nature of Euclidean factors for a geodesic ball in the Gromov-Hausdorff sense.
\begin{theorem}[Almost Splitting Theorem \cite{CC96}]\label{SplittingThm}
	Given $n\ge 2,\epsilon>0$, there exists a constant $\delta=\delta(n,\epsilon)>0$, to the following effect. Let $(M,p)$ be an $n$-manifold with Ricci curvature $\Ric_{M}\ge-(n-1)\delta$.
	\begin{enumerate}
		\item If there exists a $(\delta,k)$-splitting map $u: B_2(p)\to\R^k$, then $B_1(p)$ is $\epsilon$-Gromov-Hausdorff close to $B_1((0^k,x^*))\subset\R^k\times X$ for some length space $X$.
		
		\item If $B_{\delta^{-1}}(p)$ is $\delta$-Gromov-Hausdorff close to $B_{\delta^{-1}}((0^k,x^*))\subset\R^k\times X$, then there exists an $(\epsilon,k)$-splitting map $u:B_1(p)\to \R^k$.
	\end{enumerate}

\end{theorem}

The following transformation theorem illustrates how the stability of almost $k$-Euclidean at a point under different scales influences the almost splitting map.

\begin{theorem}[Transformation Theorem \cite{CJN21,HH22}]\label{TransThm}
	For any $n,\epsilon>0$ and $\eta>0$, there exists $\delta_{0}=\delta_0(n, \epsilon, \eta)>0$ such that, for every $\delta\in(0,\delta_{0}]$, the following holds.
	Suppose $(M,p)$ is an $n$-manifold with $\Ric_{M}\ge-(n-1)\delta$, and there is an $s\in(0,1)$ such that, for a $k\in[1,n]$ and any $r\in[s,1]$, $B_{r}(p)$ is $(\delta, k)$-Euclidean but not $(\eta, k+1)$-Euclidean. Let $u : (B_{1}(p),p)\rightarrow(\mathbb{R}^{k'},0^{k'})$ be a $(\delta, k')$-splitting map for some integer $k'\in[1,k]$. Then for each $r\in[s,1]$, there exists a $k'\times k'$ lower triangle matrix $T_{r}$ with positive diagonal entries such that $T_{r}u : B_{r}(p)\rightarrow\mathbb{R}^{k}$ is an $(\epsilon,k')$-splitting map.
\end{theorem}

\section{Proof of Theorem \ref{MainThm}}\label{ProfMain}

First, we observe the following fact of Euclidean geometry.

\begin{lemma}\label{PolarnessLemma}
	Let $(\R^k\times X,(0^k,x^*))$ be a metric product space of $\R^k$ and a proper length space $X$, satisfying that, $X$ contains no line, and any isometry of $X$ fixes $x^*$, and let $G$ be a nilpotent closed subgroup of isometries of $\R^k\times X$. If there is a tangent cone of the quotient space $(Y,p):=(\R^k\times X,\bar p)/G$ at $p$ containing no line, where $\bar p=(0^k,x^*)$, then the orbit $G(\bar p)$ is isometric to $\R^s$ for some integer $s\in[0,k]$. 
\end{lemma}

\begin{remark}
	
	If we only prove the finite generation of $\pi_1(M)$ as stated in Theorem \ref{MainThm}, then assuming that $G$ is abelian in Lemma \ref{PolarnessLemma} is sufficient, significantly simplifying the proof of the lemma. The case where $G$ is nilpotent in Lemma \ref{PolarnessLemma} is employed to establish that $\pi_1(M)$ in Theorem \ref{MainThm} is virtually abelian and prove Theorem \ref{MisalsmotkPolar}.
\end{remark}	
\begin{remark}
	The nilpotent assumption of $G$ is indispensable in Lemma \ref{PolarnessLemma}. For example, let $(\R^1,0)$ be the one-dimensional Euclidean space and $G$ be the subgroup of isometries of $\R^1$ generated by reflection about the origin and the translation by $1$, which is not nilpotent. Then the quotient space $(\R^1,0)/G$ is isometric to the interval $([0,\frac12],0)$, whose tangent cone at $0$ contains no line. However, the orbit $G(0)$ is the set of integers.

\end{remark}

\begin{remark}\label{PoleRemark}
	If we assume $X$ has a pole at $x^*$ in Lemma \ref{PolarnessLemma}, then we conclude that $(Y,p)$ has a pole at $p$. This fact is used in the proof of Theorem \ref{MisalsmotkPolar}.
\end{remark}

\begin{proof}
By the condition that $X$ contains no line and any isometry of $X$ fixes $x^*$, it yields that the isometry group, $\Isom(X)$, is compact and $\Isom(\R^k\times X)=\Isom(\R^k)\times\Isom(X)$. So $G$ induces a (not necessary effective) close isometric action on $\R^k=\R^k\times \{x^*\}$. Hence $G(\bar p)$ is isometric to $G(0^k)$. And since we have a canonical submetry $(Y,p)\to(\R^k/G,\bar 0^k)$, induced from the projection $(\R^k\times X,\bar p)\to(\R^k,0^k)$, applying the condition that there is a tangent cone of $Y$ at $p$ containing no line, it yields that the tangent cone of $\R^k/G$ at $\bar 0^k$ contains no line. Therefore it suffices to verify the case where $X$ is a single point space. That is, 

\begin{sublemma}
	Let $G$ be a closed nilpotent subgroup of $\Isom(\R^k)$. If the tangent cone of $\R^k/G$ at $\bar 0^k$ contains no line, where $\bar 0^k$ is the image of $0^k$ under the quotient map $\R^k\to\R^k/G$, then the orbit $G(0^k)$ is isometric to $\R^s$ for some integer $s\in[0,k]$. Consequently, $\R^k/G$ has a pole at $\bar 0^k$.     
\end{sublemma}

	Let $H:=G_{0^k}$ be the isotropy group of $G$ at $0^k$, which is equivalent to that $H$ is a subgroup of $O(k)$. Since $G$ is a closed nilpotent subgroup of $\Isom(\R^k)$, the identity component $G_0$ must be central in $G$ (see \cite[Lemma 1.4]{Pan22b}). Therefore, for any $A\in H\subset O(k)$, $f\in G_0$, if $A(x)=x$, then $A(f(x))=f(A(x))=f(x)$. That is, if $x\in\ker(I-A)$, then $f(x)\in\ker(I-A)$. Let $T:=\cap_{A\in H}\ker(I-A)$ which is of course a linear subspace, $\R^s$, of $\R^k$. By the discussion above, $G_0$ preserves $T$ and $H$ acts on $T$ trivially, and the effective action induced by $G_0$ on $T$ has no isotropy at $0^k\in T$. By the condition that the tangent cone of $\R^k/G$ at $\bar 0^k$ contains no line, together with the slice theorem, we obtain that $G_0(0^k)=T$.

	What left is to verify $G(0^k)\subset T$. Consider the projection $(\R^k,0^k)\to(\R^k/\spa{G_0,H},\hat 0^k)$, where $\spa{G_0,H}$ is the subgroup generated by $G_0,H$. Observe that the tangent cone of $\R^k/\spa{G_0,H}$ at $\hat 0^k$ is isometric to that of $\R^k/G$ at $\bar 0^k$, since the limits of $\spa{G_0,H}$-action and $G$-action coincide when blowing up $\R^k$ at $0^k$. Consequently, the tangent cone of $\R^k/\spa{G_0,H}$ at $\hat 0^k$ contains no line. Note that by the fact $\spa{G_0,H}(0^k)=T$ which is a linear subspace, $\hat 0^k$ is a pole of $\R^k/\spa{G_0,H}$. Now we conclude that any isometry of $\R^k/\spa{G_0,H}$ fixes $\hat 0^k$.  
	
	Let $H_0:=\spa{G_0,H}$, and $H_{i+1}:=N_G(H_i)$ be the normalizer of $H_i$ in $G$. Since $G$ is nilpotent, there exists a finite $l\ge0$ such that $H_l=G$. Thus we have a series of submetries 
	
	\begin{equation*}\label{Submetries}
		(\R^k/\spa{G_0,H},\hat 0^k )=(\R^k/H_0,p_0)\to(\R^k/H_1,p_1)\to...\to(\R^k/H_l,p_l)=(\R^k/G,\bar 0^k).
	\end{equation*}

By an induction on $i=0,1,2,...,l-1$, we obtain that any isometry of $\R^k/H_i$ fixes $p_i$. In particular, $H_{i+1}/H_i$ fixes the basic point $p_i\in\R^k/H_i$, which implies $H_{i+1}(0^k)\subset H_i(0^k)$. Consequently, $G(0^k)=H_l(0^k)\subset H_{l-1}(0^k)\subset...\subset H_0(0^k)= G_0(0^k)=T$, which completes the proof.

\end{proof}
	A byproduct of Lemma \ref{PolarnessLemma} is the following. In order not to interrupt the proof of Theorem \ref{MainThm}, we have arranged its proof in the Appendix \ref{SubsectionByproduct}.
\begin{corollary}\label{Byproduct}
	Let $(M,p)$ be a complete pointed open $n$-manifold with nonnegative Ricci curvature and nilpotent fundamental group. If $\limsup_{r\to\infty}\diam (\partial B_r(p))<2r$ (respect to extrinsic distance) and the universal cover $\tilde M$ has Euclidean volume growth, i.e., for $\tilde p\in\tilde M$, $\liminf_{r\to\infty}r^{-n}\vol(B_r(\tilde p))>0$, then $\pi_1(M)$ is finitely generated and virtually abelian.
\end{corollary}

The second key observation claims that if $(\tilde M,\tilde p)$ with nonnegative Ricci curvature satisfies Definition \ref{DefAlmostk-Polar} (1), then the number of almost Euclidean factors of $B_r(p)$, in the Gromov-Hausdorff sense, is non-increasing respect to $r\ge R$.

\begin{lemma}\label{StablityLemma}
	For any integer $n\ge 2$ and small $\eta,\epsilon>0$, there exists a constant $\delta=\delta(n,\eta,\epsilon)$, satisfying the following properties. Let $(M,p)$ be a pointed open $n$-manifold with nonnegative Ricci curvature and $(\tilde M,\tilde p)\to(M,p)$ be the Riemannian universal cover. If there exist constants $R>0$, and integer $k\in[0,n]$, such that, for any $r\ge R$, $B_r(\tilde p)$ is $(\delta,k)$-Euclidean and not $(\eta,k+1)$-Euclidean, and there exist $R_0> R$ and integer $k'\in[0,k]$ such that $B_{R_0}(p)$ is $(\delta,k')$-Euclidean, then for any $r\in [R,R_0]$, $B_r(p)$ is $(\epsilon,k')$-Euclidean.
\end{lemma}
\begin{proof}
	 Without loss of generality, we may assume $k'\ge1$ and $\delta\ll1$. By restricting a $\delta R_0$-Gromov-Hausdorff approximation from $B_{R_0}(p)\to B_{R_0}(0^{k'})$ to $B_r(p)$ and slightly modifying the image, one can observe that for any $r\in[\frac16\sqrt\delta R_0,R_0]$, $B_r(p)$ is $(60\sqrt\delta,k')$-Euclidean. In the following, we always assume $60\sqrt\delta<\epsilon$. Hence we just need to verify the case that $R\notin[\frac16\sqrt\delta R_0,R_0]$, that is, $R_0\ge \frac6{\sqrt\delta}R$. 
	 
	 By a scaling of the distance on $M$ by $(\sqrt\delta R_0)^{-1}$, the condition that $B_{R_0}(p)$ is $(\delta,k')$-Euclidean is equivalent to
	 $d_{GH}(B_{\sqrt\delta^{-1}}(p,(\sqrt\delta R_0)^{-1}M),B_{\sqrt\delta^{-1}}(0^{k'},x^*))\le\sqrt\delta$ for some $(\R^k\times X,0^k,x^*)$. By Cheeger-Colding's almost splitting theorem (Theorem \ref{SplittingThm}), there exists a small constant $\delta(n)>0$ such that if $\delta\le\delta(n)$, then there is a $(\Psi(\delta|n),k')$-splitting map, $\bar u:B_{1}(p,(\sqrt\delta R_0)^{-1}M)\to\R^{k'}$. Let $u:=\sqrt\delta R_0\bar u$. Note that $u:B_{\sqrt\delta R_0}(p)\to\R^{k'}$ is also a $(\Psi(\delta|n),k')$-splitting map. 
	 
	 By a covering lemma, \cite[Lemma 5.3]{HH22}, the lifting map $\tilde u:=u\circ \pi:B_{\sqrt\delta R_0}(\tilde p)\to\R^{k'}$ is a $(C(n)\Psi(\delta|n),k')$-splitting map, where $\pi:(\tilde M,\tilde p)\to (M,p)$ is the covering projection. Now by the condition that for any $r\in[R,\sqrt\delta R_0]$, $B_r(\tilde p)$ is $(\delta,k)$-Euclidean and not $(\eta,k+1)$-Euclidean, we apply the transformation theorem (Theorem \ref{TransThm}), which concludes that, for any $r\in[R,\sqrt\delta R_0]$, there exists a $k'\times k'$ lower triangle matrix $T_r$ with positive diagonal entries such that $T_r\tilde u:B_r(\tilde p)\to\R^{k'}$ is a $(\Psi_1(\delta|n,\eta),k)$-splitting map. By \cite[Lemma 5.3]{HH22} again, for $r\in[R,\frac13\sqrt\delta R_0]$, $T_ru:B_{r}(p)\to\R^{k'}$ is a $(C(n)\Psi_1(\delta|n,\eta),k')$-splitting map. Now applying Cheerger-Colding's almost splitting theorem again, for any $r\in[R,\frac16\sqrt\delta R_0]$, $B_r(p)$ is $(\Psi_2(\delta|n,\eta),k')$-Euclidean. So we just need to choose $\delta$ small enough such that $\Psi_2(\delta|n,\eta)\le\epsilon$ which meets the requirement.

\end{proof}
Now we are ready to prove Theorem \ref{MainThm}.

\begin{proof}[Proof of Theorem \ref{MainThm}]
	By Wilking's reduction (Theorem \ref{WilkingsReduction}), we assume $\pi_1(M)$ is abelian. The proof is argued by contradiction. Suppose that there exist $n,\eta>0,\delta_i\to0$, and a sequence of $n$-manifolds $(M_i,p_i)$ with nonnegative Ricci curvature and abelian fundamental group, satisfying that the universal cover $\tilde M_i$ is $(\delta_i,\eta,k)$-polar at infinity with respect to $\tilde p_i$, and $\pi_1(M_i)$ is not finitely generated. By a scaling, we may assume $R=1$ in Definition \ref*{DefAlmostk-Polar} for any $(\tilde M_i,\tilde p_i)$. For each $i$, fix a Gromov's short basis $\{\gamma_{i,1},\gamma_{i,2},...\}\subset\pi_1(M_i,p_i)$. Let $r_{i,j}=|\gamma_{i,j}|$, the length of shortest presentation of $\gamma_{i,j}$. 
	
	Now for each $i$, by assumption, $\lim_{j\to\infty}r_{i,j}=\infty$. Therefore we can choose integers $1\le j_i^0<j_i^1<...<j_i^{k+1}$ satisfying,
	\begin{equation}\label{BasisLength}
		1<r_{i,j_i^s}<i^{-1}r_{i,j_i^{s+1}},\,s=0,1,...,k.
	\end{equation}
  	Then by conditions, up to a subsequence, for each $s\in\{0,1,...,k+1\}$, we have the following commutative diagram of equivariant Gromov-Hausdorff convergence,
	\begin{equation}\label{Digram1}
		\xymatrix{
			(r_{i,j_i^s}^{-1}\tilde M_i,\tilde p_i,\Gamma_i) \ar[rr]^{GH}\ar[d]_{}&&(\R^k\times X_s,(0^k,x_s^*),G) \ar[d]^{} \\
			(r_{i,j_i^s}^{-1} M_i,p_i)\ar[rr]^{GH}&  & (\R^{t_s}\times Y_s,(0^{t_s},y_s^*)),}
	\end{equation}
	where each $\Gamma_i$ is the deck transformation, $G$ is the limit abelian closed group, and $X_s$ and $Y_s$ contain no line. Note that the $\R^{t_s}$-factor of $\R^{t_s}\times Y_s$ can be lifted to an $\R^{t_s}$-factor of $\R^k\times X_s$, on which $G$ acts trivially. Specially, $t_s\le k$.

	We claim: $t_{s}+1\le t_{s-1}$ for $s=1,...,k+1$. Assuming this claim, we conclude $t_0\ge k+1$ which is impossible. Hence what remain is to verify the claim. 
	
	The remaining discussion is for a fixed $s\in\{1,...,k+1\}$. 
	
	By condition (2) of Definition \ref{DefAlmostk-Polar}, $(0^k,x_s^*)$ is a pole of $\R^k\times X_s$. Hence any isometry of $X_s$ fixes $x^*_s$. And by Lemma \ref{Nonconnectness}, the obit $G((0^k,x_s^*))$ is not connected. Thus we can apply the contrapositivity of Lemma \ref{PolarnessLemma} to the projection $(\R^{k-t_s}\times X_s^*,(0^{k},x_s^*),G)\to(Y_s,y_s^*)$, we conclude that every tangent cone of $Y_s$ at $y_s^*$ splits an $\R$-factor. We choose $\lambda_1,\lambda_2,...\lambda_l,...,\to\infty$ such that $(\lambda_lY_s,y_s^*)\GH(\R\times \bar Y,\bar y)$. For each integer $l>0$, choose integer $i_l$ satisfying that ${i_l}>\lambda_l$ and $$d_{GH}((r_{i_l,j_{i_l}^s}^{-1}M_{i_l},p_{i_l}),(\R^{t_s}\times Y_s,(0^{t_s},y_s^*)))<l^{-1}\lambda_l^{-1}.$$

 	Putting $j_l^s:=j_{i_l}^s$, we have 
	$$(r_{i_l,j_l^s}^{-1}\lambda_lM_{i_l},p_{i_l})\GH(\R^{t_s+1}\times \bar Y,(0^{t_s+1},\bar y)).$$ Specially for each $L\ge1$, $B_{Lr_{i_l,j_l^s}\lambda_l^{-1}}(p_{i_l})\subset M_{i_l}$ is $(\epsilon_l(L),t_s+1)$-Euclidean for some $\epsilon_l(L)\to0$ as $l\to\infty$.
	
	Noting that we have,
	$$r_{i_l,j_l^s}^{-1}\lambda_l<i_l^{-1}r_{i_l,j_l^{s-1}}^{-1}\lambda_l<r_{i_l,j_l^{s-1}}^{-1}<1,$$ where the first inequality is by (\ref{BasisLength}) and the second one is by ${i_l}>\lambda_l$. Using the above inequality, we apply Lemma \ref{StablityLemma} to each $(M_{i_l},p_{i_l})$ with $R_0:=Lr_{i_l,j_l^s}\lambda_l^{-1}>Lr_{i_l,j_l^{s-1}}>L=:R$, which implies $B_{Lr_{i_l,j_l^{s-1}}}(p_{i_l})\subset M_{i_l}$ is $(\Psi(\epsilon_l(L),\delta_{i_l}|n,\eta),t_s+1)$-Euclidean. So,

	 $$(r_{i_l,j_l^{s-1}}^{-1}M_{i_l},p_{i_l})\GH(\R^{t_s+1}\times \bar Y',(0^{t_s+1},\bar y')),$$ for some length space $(\bar Y',\bar y')$. Comparing the convergent sequence above and the one in diagram (\ref{Digram1}), we obtain that $\R^{t_{s-1}}\times Y_{s-1}$ is isometric to $\R^{t_s+1}\times \bar Y'$. Combining the fact that $Y_{s-1}$ does not split any $\R$-factor, we conclude the claim $t_s+1\le t_{s-1}$, which finishes the proof. 
	
\end{proof}

\begin{proof}[Proof of Corollary \ref{Largek}]\quad
	Put $\bar\eta(n):=\frac12d_{GH}(B_1(0^n),B_1(0^{n+1}))$.

	(1) We claim that, for any $\epsilon>0$, there exists an $\eta(n,v,\epsilon)>0$, such that, if $B_{r_i}(\tilde p)$ is $(\eta(n,v,\epsilon),n-1)$-Euclidean for some $r_i\to\infty$, then $B_r(\tilde p)$ is $(\epsilon,n)$-Euclidean for any $r>0$. 
	
	The claim is based on the combination of two theorems: the volume convergence theorem \cite{Co97}, and the almost metric cone implies almost volume cone theorem \cite{CC96}. Suppose the opposite; let $(\tilde M_i,\tilde p_i)$ be a sequence that contradicts this claim. Then each $(\tilde M_i,\tilde p_i)$ has a tangent cone at infinity $(Y_i,y_i)$ that is a metric cone with the vertex $y_i$ satisfying that $B_1(y_i)$ is $(\eta_i,n-1)$-Euclidean for some $\eta_i\to0$ and $\vol(B_1(y_i))\ge v$. Up to a subsequence, we may assume $(Y_i,y_i)\GH(Y,y)$ where $Y$ is a metric cone with vertex $y$ which splits $\R^{n-1}$. According to the codimension $2$ theorem of singularity of non-collapsed Ricci limit space \cite{CC97}, $Y$ is isometric to $\R^n$. Therefore, $\vol(B_1(y_i))\ge (1-\epsilon_i)\vol(B_1(0^n))$ for some $\epsilon_i\to0$. Hence we obtain $B_r(\tilde p_i)$ is $(\Psi(\epsilon_i|n),n)$-Euclidean for each $r>0$, which yields a contradiction. The claim holds.
	
	Let $\epsilon(n)$ be determined by Theorem \ref{MainThm} with respect to $n,\bar\eta(n)$. By the above claim, putting $\eta(n,v):=\eta(n,v,\epsilon(n))$, if $B_{r_i}(\tilde p)$ is $(\eta(n,v),n-1)$-Euclidean for some $r_i\to\infty$, then $\tilde M$ satisfies $(\epsilon(n),\bar\eta(n),n)$-polar at infinity, which implies $\pi_1(M)$ is finitely generated.  
	
	Thus, we can assume that $B_r(\tilde p)$ is not $(\eta(n,v),n-1)$-Euclidean for any sufficiently large $r>0$. Combining this with the conditions, we obtain that $\tilde M$ satisfies $(\delta_1,\eta(n,v),n-2)$-polar at infinity. Hence if we choose $\delta_1=\delta_1(n,v)$ determined by Theorem \ref{MainThm} with respect to $n-1,\eta(n,v)$, then $\pi_1(M)$ is finitely generated.
	
	(2) Let $\eta_0(n)$ be determined by Theorem \ref{MainThm} with respect to $n,\bar \eta(n)$. And by the quantitative volume rigidity theorem \cite[Theorem 0.8]{Co97}, we can choose small $\eta_1(n)$ satisfying that if $B_{r_i}(\tilde p)$ is $(\eta_1(n),n)$-Euclidean for some $r_i\to\infty$, then $B_r(\tilde p)$ is $(\eta_0(n),n)$-Euclidean for any $r>0$.
	
	Thus if $B_{r_i}(\tilde p)$ is $(\eta_1(n),n)$-Euclidean for some $r_i\to\infty$, then $\pi_1(M)$ is finitely generated by Theorem \ref{MainThm}. Therefore, we can assume $B_{r}(\tilde p)$ is not $(\eta_1(n),n)$-Euclidean for any sufficiently large $r$. Moreover, by the condition, it is direct to verify that for any sufficiently large $r>0$, and any $x\in\partial B_r(\tilde p)$, $d(x,\partial B_{2r}(\tilde p))\le(1+12\delta_2 )r$. Combining these, we conclude that $\tilde M$ satisfies $(12\delta_2,\eta_1(n),n-1)$-polar at infinity. So we can then choose $12\delta_2$ to be less than the constant $\delta(n)$ determined by Theorem \ref{MainThm} with respect to $n,\eta_1(n)$. This implies that $\pi_1(M)$ is finitely generated. 
	
\end{proof}

\section{Proof of Theorem \ref{MisalsmotkPolar} and \ref{MaximalVolGrowth}}\label{ProfotherThms}
	
	\begin{proof}[Proof of Theorem \ref{MisalsmotkPolar}]

		By Theorem \ref{MainThm}, $\pi_1(M)$ is finitely generated. Combining this with Kapovitch-Wilking's Magulis lemma \cite{KW11}, we can assert the existence of a finite cover $\hat M$ of $M$ with index $\le C(n)$ and whose fundamental group is nilpotent with nilpotency step $\le n$, i.e., a lower central series of length $\le n$. It suffices to prove the conclusion for $\hat M$. Therefore, we may assume below that $\pi_1(M)$ is nilpotent with nilpotency step $\le n$.

		Argue by contradiction. Let $(M_i,p_i)$ be a sequence of open $n$-manifolds with nonnegative Ricci curvature satisfying that, for each $i$, $\pi_1(M_i)$ is nilpotent with nilpotency step $\le n$, and the Riemannian universal cover $(\tilde M,\tilde p_i)\to (M_i,p_i)$ is $(\delta_i,\eta,k)$-polar at infinity for $\delta_i\to0$ (by a scaling, we assume $R=1$ in Definition \ref{DefAlmostk-Polar}). However, there exists a tangent cone at infinity $(Y_i,y_i^*)$, of $(M_i,p_i)$, such that for any length space $(Z,z)$ with a pole at $z$, 
		\begin{equation}\label{GaptoPolar}
			d_{GH}((Y_i,y_i^*),(Z,z))>\epsilon.
		\end{equation}
		
		Suppose $(\tilde Y_i,\tilde y_i^*,G_i)$ is an equivariant tangent cone at infinity of $(\tilde M_i,\tilde p_i,\pi_1(M_i))$, arising from the same scaling sequence as $(Y_i,y_i^*)$, i.e., there exists $r_j\to0$ such that $(r_jM_i,p_i)\GH(Y_i,y_i^*)$ and $(r_j\tilde M_i,\tilde p_i,\pi_1(M_i))\GH(\tilde Y_i,\tilde y_i^*,G_i)$. Up to a subsequence, by the assumption of almost $k$-polar at infinity, we have the following commutative diagram,
		
		\begin{equation*}\label{Diag3}
			\xymatrix{
				(\tilde Y_i,\tilde y_i^*,G_i) \ar[rr]^{GH}\ar[d]_{}&&(\R^k\times \tilde X,(0^k,\tilde x),G) \ar[d]^{} \\
				(Y_i,y_i^*)\ar[rr]^{GH}&  & (\R^{t}\times X,(0^{t},x)),}
		\end{equation*}
		where $\tilde X,X$ contain no line and $\tilde X$ has a pole at $\tilde x$. By the assumption that $\pi_1(M_i)$ are all nilpotent with nilpotency step $\le n$, so are $G_i$ and $G$. 
		
		We claim, any tangent cone of $X$ at $x$ splits no $\R$-factor. Otherwise, for any $\epsilon'>0$, there exists small $r'>0$, such that for any large $i$, $B_{r'}(y_i^*)$ is $(\delta_i'+\epsilon',t+1)$-Euclidean for some $\delta_i'\to0$. Then fixed such an $i$, there exists $r_j\to\infty$, such that $B_{r_j}(p_i)$ is $(2(\delta_i'+\epsilon'),t+1)$-Euclidean. Now applying Lemma \ref{StablityLemma} to $(M_i,p_i)$ for $R_0=r_j$ and $R=1$, it yields that for any $r\ge 1$, $B_r(p_i)$ is $(\Psi(\delta_i'+\epsilon'|n,\eta),t+1)$-Euclidean. Since $\epsilon'$ is arbitrary, we conclude that $\R^t\times X$ splits $\R^{t+1}$-factor, which is a contradiction to that $X$ contains no line. Hence our claim follows.
		
		Now we apply Lemma \ref{PolarnessLemma} and Remark \ref{PoleRemark} to $(\R^{k-t}\times\tilde X,(0^{k-t},\tilde x),G)\to(X,x)$ to yield that $X$ has a pole at $x$. This contradicts to assumption (\ref{GaptoPolar}).

	\end{proof}

To prove Theorem \ref{MaximalVolGrowth}, we need the following two lemmas. The first one, which is implied in \cite{CFG}, is a classical fact in collapsing theory. For the reader's convenience, we provide a sketched proof in Appendix \ref{Subsectioncontracttoinfranil}.

\begin{lemma}\label{LocalConctraction}
	For $n,\epsilon>0$, there exist universal positive constants $r=r(n)\in(0,1)$ and $\delta=(n,\epsilon)$ to the following effect. Suppose that a pointed $n$-manifolds $(M,p)$ satisfies that the sectional curvature is bounded, $|\sec_{M}|\le 1$, $\overline{B_{8}(p)}$ is compact and $\partial B_8(p)$ is not empty. If $\vol(B_1(p))<\delta$, then there exist an embedding infra-nilmanifold $F\subset M$ containing $p$, and an open regions $U\subset M$, satisfying that, $\diam F\in(0,\epsilon)$, $B_{e^{-\epsilon}r}(p)\subset U\subset B_{e^\epsilon r}(p)$, and $U$ is diffeomorphic to the normal bundle of $F$. 
\end{lemma}

The second tool we require is the following theorem regarding local Ricci flow.

\begin{theorem}\cite{HW22,HRW20}\label{LocalSmoothing}
	For any $\alpha,\rho\in(0,1]$ there are $\delta=\delta(n,\alpha,\rho),\epsilon=\epsilon(n,\alpha,\rho)\in(0,1)$ such that if $(M,g)$ is a complete manifold with $\Ric\ge-(n-1)$, and $p\in M$ satisfies $d_{GH}(B_{2\rho}(\tilde p),B_{2\rho}(0^n))\le2\rho\delta$, where $\tilde p$ is a lift of $p$ in the Riemannian universal covering space of $B_{2\rho}(p)$, then there is a Ricci flow solution with initial data $(B_\rho(p,g),g)$, that exists up to time $\epsilon^2$ and for any $t\in(0,\epsilon^2]$, the curvature satisfies $$\sup_{B_\rho(p,g)}\|\mathrm{Rm}_{g(t)}\|_{g(t)}\le\frac{\alpha}{t}.$$
	
\end{theorem}

\begin{proof}[Proof of Theorem \ref{MaximalVolGrowth}]
		
	By Cheeger-Colding's Reifenberg type theorem (\cite{CC97}), $\tilde M$ is diffeomorphic to $\mathbb{R}^n$. Consequently, $\pi_1(M)$ is torsion-free; otherwise, the $\pi_1(M)$-action on $\tilde M\approx\R^n$ cannot be free. This, combined with Li's result (\cite{Li86}), implies
	\begin{equation}\label{CollapsingofM}
		 \lim_{r\to\infty}\frac{\mathrm{vol}(B_{r}(p))}{r^n}=0
	\end{equation}

	Again by Cheeger-Colding's result (\cite{CC96}), for $\delta$ sufficiently small depending on $n$, $d_{GH}(B_r(\tilde p),B_r(0^n)) \leq r\Psi(\delta|n)$ for any $r>0$. Hence Theorem \ref{MisalsmotkPolar} can be applied to $M$, which shows that, for $x\in M$ outside a compact subset, $\frac{e_p(x)}{r_p(x)}\le\Psi(\delta|n)$, where $e_p(x):=r_p(x)+d(x,\partial B_{2r_p(x)}(p))-2r_p(x)$, and $r_p(x):=d(x,p)$. Now we can apply \cite[Theorem 1.5]{Hua22} for $\alpha=1$ to conclude that there is a bounded open subset $V\ni p$ such that $M$ deformation retracts to $\overline V$.

	Fixing a sufficiently large $r$, which will be determined later, and given any small $\alpha\in(0,1)$, applying Theorem \ref{LocalSmoothing} to the $2$-ball $B_{2}(p,r^{-2}g)$ with respect to the metric $r^{-2}g$, where $g$ is the original metric tensor on $M$, if $\delta$ is small depending on $n$ and $\alpha$, then there exists a (non-complete) Ricci flow solution $g(t)$ with initial data $(B_{1}(p,r^{-2}g),r^{-2}g)$ existing for at least up to time $\epsilon(n,\alpha)^2$, and satisfying for all $t \in (0,\epsilon(n,\alpha)^2]$:
	$$
	\sup_{B_1(p,r^{-2}g)}\|\mathrm{Rm}_{g(t)}\|_{g(t)} \leq \frac{\alpha}{t}.
	$$

	Put $\bar g:=g(\epsilon(n,\alpha)^2)$. Using the distance distortion estimate \cite[Lemma 1.11]{HKRX} (see also \cite{HRW20}) and the standard volume evolution of Ricci flow in conjunction with (\ref{CollapsingofM}), for sufficiently small $\epsilon(n,\alpha)$, we have $$B_{0.7}(p,r^{-2}g)\subset B_{0.8}(p,\bar g)\subset B_{0.9}(p,r^{-2}g),$$
	$$\vol_{\bar g}(B_{0.8}(p,\bar g))\to 0,\,\text{as }r\to\infty.$$
	
	Hence for $r$ sufficiently large, up to a scaling, we can apply Lemma \ref{LocalConctraction} to $(B_{0.8}(p,\bar g),p)$ to yield that there is an $s=s(n,\epsilon(n,\alpha))\in(0,0.1)$, and an open region $U$, satisfying that, $B_{0.9s}(p,\bar g)\subset U\subset B_{1.1s}(p,\bar g)$, and $U$ is diffeomorphic to the normal bundle of an embedding infra-nilmanifold $p\in F\subset U$. Again using the distance distortion estimate \cite[Lemma 1.11]{HKRX}, for sufficiently small $\epsilon(n,\alpha)$, we obtain $$B_{0.5sr}(p,g)\subset U\subset B_{2sr}(p,g).$$

	Given that $r$ is sufficiently large to ensure that $\overline{V}\subset B_{0.5sr}(p,g)$, which is in turn contained within $U$, we can deduce that $M$ deformation retracts to $F$ by compositing a deformation retraction from $M$ to $\overline{V}$ with a deformation retraction from $U$ to $F$. 
	
	What is left is to show that $F$ is diffeomorphic to a flat manifold. Note that, $\pi_1(M)$ is virtually abelian (as discussed in Subsection \ref{virabelian}), which implies the same for $\pi_1(F)$. By Malcev's rigidity theorem (\cite[Theorem 3.7]{CFG}), there exists an isomorphism from the universal cover of $F$ to $\R^n$. Hence $\pi_1(F)$ is a subgroup of $\mathrm{Aff}(\R^n)$. Using the classical theorem that any compact subgroup of $GL(\R^n)$ is conjugate to a subgroup of $O(n)$, we conclude that $\pi_1(F)$ is conjugate to a subgroup of $\Isom(\R^n)$. This implies that $F$ is diffeomorphic to a flat manifold, thus completing the proof.

	\end{proof}

\section{appendix}

\subsection{$\pi_1(M)$ is virtually abelian}\label{virabelian}
In this subsection, we relies on the escape rate gap theorem \cite[Theorem A]{Pan22}, which asserts that if an open $n$-manifold $M$ with nonnegative Ricci curvature has escape rate less than some positive constant $\epsilon(n)$, then $\pi_1(M)$ is virtually abelian. 

Let $(M,p)$ be a pointed open manifold with nonnegative Ricci curvature and $G$ a closed subset of the isometry group of $M$. We use $\Omega(M,G)$ to denote all equivariant tangent cones at infinity of $(M,p)$. That is, for any $(X,x,H)\in\Omega(M,G)$, there exists a convergent sequence $(r_i^{-1}M,p,G)\GH(X,x,H)$ for some $r_i\to\infty$ as $i\to\infty$. Firstly, we need the connectness of $\Omega(M,G)$,
\begin{lemma}\label{Connectnessofinfinitycone}
	The set $\Omega(M,G)$ is connected in the following sense. Given $(X_m,x_m,H_m)\in \Omega(M,G)$, $m=0,1$, for any $\epsilon>0$, there exist $(Y_s,y_s,G_s)\in \Omega(M,G)$, $s=0,1,...,S$, such that $(Y_0,y_0,G_0)=(X_0,x_0,H_0)$ and $(Y_S,y_S,G_S)=(X_1,x_1,H_1)$, and $$d_{GH}((Y_s,y_s,H_s),(Y_{s+1},y_{s+1},H_{s+1}))\le\epsilon,\,s=0,1,...,s-1.$$
\end{lemma}
This property is well known for experts. For the reader's convenience, we include a proof here.
\begin{sublemma}\label{WeakStable}
	Let $(M,p)$ be a pointed open manifold with nonnegative Ricci curvature and $G$ a closed subgroup of the isometry group of $M$. Then for any $\epsilon>0$, there exists $R>0$, such that for any $r\ge R$, there exists $(X,x,H)\in\Omega(M,G)$ satisfying that $d_{GH}((r^{-1}M,p,G),(X,x,H))\le\epsilon$.

\end{sublemma}
\begin{proof}
	If not, then there exist $\epsilon_0>0$, $r_i\to\infty$, such that for any $(X,x,H)\in\Omega(M,G)$ satisfying that $d_{GH}((r_i^{-1}M,p,G),(X,x,H))>\epsilon_0$. By the compactness criterion for equivariant Gromov-Hausdorff convergence, this is impossible.
\end{proof}

\begin{proof}[Proof of Lemma \ref{Connectnessofinfinitycone}]
	Let $R>0$ be the constant depending $\frac\epsilon3$ as in Lemma \ref{WeakStable}. And for $m=0,1$, let $r_m>R$ satisfy that $d_{GH}((r_m^{-1}M,p,G),(X_m,x_m,H_m))\le\frac\epsilon3	$ and $r_0<r_1$. For any integer $S>0$ and $s=0,1,...,S$, put $r_s=r_0+\frac{s}{S}(r_1-r_0)$. It is simple to see $$d_{GH}((r_s^{-1}M,p,G),(r_{s+1}^{-1}M,p,G))\le\delta_S,$$ where $\delta_S\to0$ as $S\to\infty$. Now the required conclusion follows from Lemma \ref{WeakStable}.
	
\end{proof}

By Theorem \ref{MainThm}, $\pi_1(M,p)$ is finitely generated. Then according to the Margulis lemma in \cite{KW11}, $\pi_1(M)$ has
a nilpotent subgroup of index $\le C(n)$ with nilpotency step at most $n$. In order to prove $\pi_1(M)$ is virtually abelian, up to a finite cover, we always assume $\pi_1(M)$ is nilpotent with nilpotency step at most $n$ below.

By the escape rate gap theorem in \cite{Pan22}, we need to verify that, for any $\epsilon>0$, there exists $\delta=\delta(n,\eta,\epsilon)$ such that $(M,p)$ satisfying conditions of Theorem \ref{MainThm} has escape rate less than $\epsilon$. Applying \cite[Proposition 4.2]{Pan22}, it suffices to show that, there exists an integer $s\ge0$, such that any $(Y,y,G)\in\Omega(\tilde M,\pi_1(M,p))$, satisfies 
\begin{equation}\label{G-OrbitisR^s}
	d_{GH}((Y,y,G(y)),(\R^s\times X,(0^s,x),\R^s\times\{x\}))\le \Psi(\delta|n,\eta),
\end{equation}
where $(X, x)$ is a length space that depends on $(Y, y)$ (which may contain a line). Firstly we show the following lemma which is a weak version of the above property.
\begin{lemma}\label{sdenpendsonY}
	For a $(Y,y,G)\in\Omega(\tilde M,\pi_1(M,p))$, there exists an integer $s\in[0,n]$ and a length space $(X,x)$ satisfying inequality (\ref{G-OrbitisR^s}).
\end{lemma}

The distinction between Lemma \ref{sdenpendsonY} and the aforementioned property lies in the fact that in Lemma \ref{sdenpendsonY}, $s$ may depend on the choice of $(Y,y,G)$ beforehand. However, the independence of $s$ on the choice of $(Y,y,G)\in\Omega(\tilde M,\pi_1(M,p))$ is inferred from Lemma \ref{Connectnessofinfinitycone} and the observation that $(\R^{s_0}\times X_0,(0^{s_0},x_0),\R^{s_0}\times\{x_0\})$ and $(\R^{s_1}\times X_1,(0^{s_1},x_1),\R^{s_1}\times\{x_1\})$ have a definite positive distance dependent solely on $n$ in pointed Gromov-Hausdorff distance if $s_0\neq s_1$. Therefore, we just need to prove Lemma \ref{sdenpendsonY}.

\begin{proof}[Proof of Lemma \ref{sdenpendsonY}]

	Argue by contradiction. Assume there exist an $\epsilon>0$, and a contradicting sequence $(M_i,p_i)$ satifying that conditions in Theorem \ref{MainThm} for $\delta_i\to0$, and for each integer $i>0$, there exists $(Y_i,y_i,G_i)\in\Omega(\tilde M_i,\pi_1(M_i,\tilde p_i))$ such that, for any length space $X'$, and any integer $s\in[0,n]$,
	
	\begin{align}\label{GapToGoodLimit}
	d_{GH}((Y_i,y_i,G_i(y_i)),(\R^{s}\times X',(0^s,x'),\R^{s}\times\{x'\}))> \epsilon.
	\end{align}
	Also by a scaling, we may assume $R=1$ in Definition \ref{DefAlmostk-Polar}.

Passing to a subsequence, we have the following commutative diagram of equivariant Gromov-Hausdorff convergence,
	\begin{equation}\label{Diag2}
		\xymatrix{
			(Y_i,y_i,G_i) \ar[rr]^{GH}\ar[d]_{}&&(\R^k\times \bar X,(0^k,\bar x),G) \ar[d]^{} \\
			(Y_i/G_i,x_i)\ar[rr]^{GH}&  & (\R^{t}\times X,(0^{t},x)),}
	\end{equation}
	where $X$ and $\bar X$ contain no line, and by the condition of $(\delta,\eta,k)$-polar at infinity, any isometry of $\bar X$ fixes $\bar x$. Recall that we have assumed each $\pi_1(M_i)$ is nilpotent with nilpotency step at most $n$. So $G$ is nilpotent. Note that by Lemma \ref{StablityLemma}, any tangent cone of $X$ at $x$ splits no $\R$-factor (The proof is similar to the penultimate paragraph of the proof of Theorem \ref{MisalsmotkPolar}). Applying Lemma \ref{PolarnessLemma} for $(\R^{k-t}\times \bar X,(0^{k-t},\bar x),G)$, it yields that $G((0^k,\bar x))$ is isometric to some $\R^{s'}\subset\R^k\times\bar X$ for some integer $s'\in[0,k-t]$. This is a contradiction.

\end{proof}

\subsection{Proof of Corollary \ref{Byproduct}}\label{SubsectionByproduct}
By Lemma \ref{Nonconnectness}, \cite[Theorem A,Proposition 4.2]{Pan22} as mentioned in Subsection \ref{virabelian}, it suffices to prove the following lemma.
\begin{lemma}
	There exists an integer $s\in[0,n]$ such that for any $(Y,y,G)\in\Omega(\tilde M,\pi_1(M))$, $(Y,y,G(y))$ is isometric to $(\R^s\times X,(0^s,x),\R^s\times\{x\})$ for some $X$.
\end{lemma}
\begin{proof}
	 Since $\tilde M$ has Euclidean volume growth, by Cheeger-Colding's theory \cite{CC96}, $(Y,y)$ is isometric to $(\R^k\times \bar X,(0^k,\bar x))$, where $\bar X$ is a metric cone with cone vertex $\bar x$, which splits no $\R$-factor. Let $r_i\to\infty$ be the sequence such that the following diagram holds, 
	 
	 \begin{equation}\label{Diag2}
	 	\xymatrix{
	 		(r_i^{-1}\tilde M,\tilde p,\pi_1(M)) \ar[rr]^{GH}\ar[d]_{}&&(\R^k\times \bar X,(0^k,\bar x),G)\ar[d]^{} \\
	 		(r_i^{-1}M,p)\ar[rr]^{GH}&  & (\underline{X},\underline{x}).} 
	 \end{equation}
	By the condition that $\limsup_{r\to\infty}\diam (\partial B_r(p))<2r$, any tangent cone of $\underline X$ at $\underline x$ splits no $\R$-factor. Since $\pi_1(M)$ is nilpotent, so is $G$. Now applying Lemma \ref{PolarnessLemma} to $(Y,y,G)=(\R^k\times \bar X,(0^k,\bar x),G)$ yields that $(Y,y,G)$ is isometric to $(\R^s\times X,(0^s,x),\R^s\times\{x\})$ for some $s$ and $X$. 
	
	Now the fact that $s$ does not depend on the choice of $(Y,y,G)$ is implied by Lemma \ref{Connectnessofinfinitycone}.
\end{proof}

\subsection{Proof of Lemma \ref{LocalConctraction}}\label{Subsectioncontracttoinfranil}

	By Abresch's result(see \cite[Theorem 1.12]{CFG}), for any $t>0$, one can smooth the original metric, $g$, of $B_7(p)$ to a new metric $\tilde g$, sastisfying that,
	\begin{enumerate}
		\item $e^{-t}g\le \tilde g\le e^{t}g$,
		\item $\|\grad-\tilde \grad \|\le t $,
		\item $\|\tilde \grad^i\mathrm{Rm}_{\tilde g}\|\le C_i(n,t)$,
	\end{enumerate}
	Moreover, at $x\in B_7(p)$, the value of $\tilde g$ depends only on $g|B_{\frac14}(x)$. Based on property (1), by fixing a small constant $t$ that depends on $\epsilon$, we can assume, without loss of generality, that $g$ itself satisfies the aforementioned estimation of higher-order derivatives of the curvature tensor (3).
	
	Now we argue by contradiction, assuming there exists an $\epsilon>0$, and a sequence of $n$-manifolds $(M_i,p_i)$ satisfying that $\| \grad^j\mathrm{Rm}_{M_i}\|\le C_j(n,t)$, $\overline{B_{8}(p_i)}$ is compact, $\partial B_8(p_i)$ is not empty, and $\vol(B_1(p_i))\to0$, while for each $i$, there are no $r>0$, $F,U\subset M_i$ meeting the requirement.
	
	By Fukaya's singular fibration theorem (\cite{Fu88}), up to a subsequence, there exists a commutative diagram,
	\begin{equation}\label{DiagSF}
		\xymatrix{
			(F(B_4(p_i)),\hat p_i,O(n)) \ar[rr]^{h_i}\ar[d]_{\pi_i}&&(Y,y,O(n))\ar[d]^{\pi} \\
			(B_4(p_i),p_i)\ar[rr]^{f_i}&  & (X,p),} 
	\end{equation}
	where $F(B_4(p_i))$ is the orthogonal frame bundle over $B_4(p)$ equipping the canonical induced metric, satisfying that 
	\begin{enumerate}
		\item $\| \grad^j\mathrm{Rm}_{F(B_4(p_i))}\|\le \bar C_j(n,t)$,
		\item $Y$ is a $C^{\infty}$-manifold,
		\item For some $\epsilon_i\to0$, each $h_i$ is an $\epsilon_i$-Gromov-Hausdorff approximation and $\epsilon_i$-Riemannian submersion, i.e., for any horizontal vector $\xi$, $e^{-\epsilon_i}|\xi|\le|\mathrm{d} h_i(\xi)|\le e^{\epsilon_i}|\xi|$, 
		\item $\|\grad ^jh_i\|\le C_j(n,t,Y)$,
		\item The restricting map $h_i:h_i^{-1}(B_{1}(y))\to B_{1}(y)$ is a smooth fiber bundle map with fiber type nilmanifolds and the norm of the second fundamental form of each $h_i$-fiber, $\|II\|\le C(n,t,Y)$,
		\item Each $f_i$ is a continuous $\epsilon_i$-Gromov-Hausdorff approximation with fiber type infra-nilmanifolds.
	\end{enumerate}

	Note that, by (2) combined with the assumption that $\partial B_8(p_i)$ is non-empty, the orbit $O(n)(y)$ is an embedding submanifold with a normal injectivity radius $>0$. Then, the properties (1-5) imply that, for all $i$, the normal injectivity radius of $h_i^{-1}(O(n)(y))$ is uniformly bounded from below by some $r\in(0,1)$, depending on $Y$ and the embedding of $O(n)(y)$ (as the proof of \cite[Proposition A2.2]{CFG} suggests). Since $\pi_i$ is a Riemannian submersion, the normal injectivity radius of $f_i^{-1}(p)=\pi_i(h_i^{-1}(O(n)(y)))$ is also uniformly bounded from below by $r$. Therefore, by (6), we can choose $F:=f_i^{-1}(p)$ and $U:=\{x\in M_i|d(x,F)<r\}$ to satisfy the required conditions, which contradicts the contradictory assumption.
	
\vspace*{20pt}

\noindent\textbf{Acknowledgments.}

The author would like to express gratitude to Prof. Jiayin Pan for his appreciation of this work and for pointing out the virtually abelianness of $\pi_1(M)$ under the assumption of Theorem \ref{MainThm}. The author would also like to thank Prof. Xian-Tao Huang and Dr. Yu Peng for the helpful comments.

\bibliographystyle{alpha}
\bibliography{ref}

\begin{thebibliography}{HKRX20}

\bibitem[And90]{An90}
M.~T. Anderson.
\newblock On the topology of complete manifolds of nonnegative {R}icci
  curvature.
\newblock {\em Topology}, 29(1):41--55, 1990.

\bibitem[BNS23a]{BAD23a}
E.~Brue, A.~Naber, and D.~Semola.
\newblock Fundamental groups and the {M}ilnor conjecture.
\newblock {\em arXiv:2303.15347}, 2023.

\bibitem[BNS23b]{BAD23b}
E.~Brue, A.~Naber, and D.~Semola.
\newblock Six dimensional counterexample to the {M}ilnor conjecture.
\newblock {\em arXiv:2303.15347}, 2023.

\bibitem[CC96]{CC96}
J.~Cheeger and T.~H. Colding.
\newblock Lower bounds on {R}icci curvature and the almost rigidity of warped
  products.
\newblock {\em Ann. of Math. (2)}, 144(1):189--237, 1996.

\bibitem[CC97]{CC97}
J.~Cheeger and T.~H. Colding.
\newblock On the structure of spaces with {R}icci curvature bounded below. {I}.
\newblock {\em J. Differential Geom.}, 45:406--480, 1997.

\bibitem[CFG92]{CFG}
J.~Cheeger, K.~Fukaya, and M.~Gromov.
\newblock Nilpotent structures and invariant metrics on collapsed manifolds.
\newblock {\em J. Amer. Math. Soc.}, 5(2):327--372, 1992.

\bibitem[CJN21]{CJN21}
J.~Cheeger, W.~Jiang, and A.~Naber.
\newblock Rectifiability of singular sets in noncollapsed spaces with {R}icci
  curvature bounded below.
\newblock {\em Ann. of Math. (2)}, 193(2):407--538, 2021.

\bibitem[CN13]{CN13}
T.~H. Colding and A.~Naber.
\newblock Characterization of tangent cones of noncollapsed limits with lower
  {R}icci bounds and applications.
\newblock {\em Geom. Funct. Anal.}, 23:134–148, 2013.

\bibitem[Col97]{Co97}
T.H. Colding.
\newblock {R}icci curvature and volume convergence.
\newblock {\em Ann. of Math. (2)}, 145(3):477--501, 1997.

\bibitem[Fuk88]{Fu88}
K.~Fukaya.
\newblock A boundary of the set of riemannian manifolds with bounded
  curvatures.
\newblock {\em J. Differential Geom.}, 28:1--21, 1988.

\bibitem[HH22]{HH22}
H.~Huang and X.-T. Huang.
\newblock Almost splitting maps, transformation theorems and smooth fibration
  theorems.
\newblock {\em arXiv:2207.10029v2}, 2022.

\bibitem[HKRX20]{HKRX}
H.~Huang, L.~Kong, X.~Rong, and S.~Xu.
\newblock Collapsed manifolds with {R}icci bounded covering geometry.
\newblock {\em Trans. Amer. Math. Soc.}, 373(11):8039–8057, 2020.

\bibitem[HRW20]{HRW20}
S.~Huang, X.~Rong, and B.~Wang.
\newblock Collapsing geometry with {R}icci curvature bounded below and {R}icci
  flow smoothing.
\newblock {\em SIGMA Symmetry Integrability Geom. Methods Appl.}, 16:123, 2020.

\bibitem[Hua22]{Hua22}
H.~Huang.
\newblock A finite topological type theorem for open manifolds with
  non-negative {R}icci curvature and almost maximal local rewinding volume.
\newblock {\em to appear in Int. Math. Res. Not. IMRN}, 2022.

\bibitem[HW22]{HW22}
S.~Huang and B.~Wang.
\newblock Ricci flow smoothing for locally collapsing manifolds.
\newblock {\em Calc. Var.}, 61:64, 4 2022.

\bibitem[KW11]{KW11}
V.~Kapovitch and B.~Wilking.
\newblock Structure of fundamental groups of manifolds with {R}icci curvature
  bounded below.
\newblock {\em arXiv:1105.5955v2}, 2011.

\bibitem[Li86]{Li86}
P.~Li.
\newblock Large time behavior of the heat equation on complete manifolds with
  nonnegative {R}icci curvature.
\newblock {\em Ann. of Math. (2)}, 124(1):1–21, 1986.

\bibitem[Mil68]{Mi68}
J.~Milnor.
\newblock A note on curvature and the fundamental group.
\newblock {\em J. Differential Geom.}, 2:1--7, 01 1968.

\bibitem[Pan19a]{Pan19b}
J.~Pan.
\newblock Nonnegative {R}icci curvature, almost stability at infinity, and
  structure of fundamental groups.
\newblock {\em arXiv:1809.10220v2}, 2019.

\bibitem[Pan19b]{Pan19a}
J.~Pan.
\newblock Nonnegative {R}icci curvature, stability at infinity and finite
  generation of fundamental groups.
\newblock {\em Geom. Topol.}, 23:3203--3231, 2019.

\bibitem[Pan20]{Pan20}
J.~Pan.
\newblock A proof of {M}ilnor conjecture in dimension 3.
\newblock {\em J. Reine Angew. Math.}, 758:253--260, 2020.

\bibitem[Pan22a]{Pan22}
J.~Pan.
\newblock Nonnegative {R}icci curvature and escape rate gap.
\newblock {\em J. Reine Angew. Math.}, 782:175--196, 2022.

\bibitem[Pan22b]{Pan22b}
J.~Pan.
\newblock Nonnegative {R}icci curvature, metric cones, and virtual abelianness.
\newblock {\em arXiv:2201.07852}, 2022.

\bibitem[Sor99]{Sor99}
C.~Sormani.
\newblock Nonnegative {R}icci curvature, small linear diameter growth and
  finite generation of fundamental groups.
\newblock {\em J. Differential Geom.}, 53:547--559, 1999.

\bibitem[Wil00]{Wil00}
B.~Wilking.
\newblock On fundamental groups of manifolds of nonnegative {R}icci curvature.
\newblock {\em Differential Geom. Appl.}, 13:129--165, 2000.

\end{thebibliography}

\end{document}